\documentclass{amsart}

\usepackage{latexsym}
\usepackage{verbatim}
\usepackage{amsmath}
\usepackage{amssymb}
\usepackage{mathrsfs}
\usepackage{amsthm}
\usepackage{tikz}
\usetikzlibrary{decorations.pathreplacing}
\usetikzlibrary{arrows.meta}

\newcommand{\Ind}{
 \setbox0=\hbox{$x$}\kern\wd0\hbox to 0pt{\hss$
 \mid$\hss}\lower.9\ht0\hbox to 0pt{\hss$\smile$\hss}\kern\wd0
}
\newcommand{\indep}[3]{
 #1\mathop{\mathpalette\Ind{}}_{#2}#3
}

\newcommand{\Notind}{
 \setbox0=\hbox{$x$}\kern\wd0\hbox to 0pt{\mathchardef
 \nn=12854\hss$\nn$\kern1.4\wd0\hss}\hbox to 0pt{\hss$\mid$\hss}\lower.9\ht0
 \hbox to 0pt{\hss$\smile$\hss}\kern\wd0
}

\usepackage{amsthm, amsfonts}
\usepackage{euscript}
\usepackage{graphicx}
\usepackage{float}
\usepackage{caption}
\usepackage{tikz-qtree}
\usepackage{subcaption}
\newcommand{\G}{\Gamma}
\renewcommand{\d}{\delta}

\renewcommand{\epsilon}{\varepsilon}

\newcommand{\C}{\mathcal{C}}
\newcommand{\K}{\mathcal{K}}

\newcommand{\E}{\mathcal{E}}

\newcommand{\conv}{\mathrm{conv}}
\newcommand{\acl}{\mathrm{acl}}
\newcommand{\dcl}{\mathrm{dcl}}

\newcommand{\val}{\mathrm{val}}
\newcommand{\tp}{\mathrm{tp}}

\newtheorem{defi}{Definition}[section]
\newtheorem{theorem}[defi]{Theorem}
\newtheorem{definition}[defi]{Definition}
\newtheorem{lemma}[defi]{Lemma}
\newtheorem{proposition}[defi]{Proposition}

\newtheorem{corollary}[defi]{Corollary}

\newtheorem{remark}[defi]{Remark}

\def\Aut{\mathop{\rm Aut}\nolimits}

\def\nil2{\mathop{\rm nil2}\nolimits}

\begin{document}
\author{Zahra Mohammadi  Khangheshlaghi \& Katrin Tent}
\thanks{This research was partially funded through the Cluster of Excellence by the
German Research Foundation (DFG) under Germany’s Excellence Strategy EXC
2044–390685587, Mathematics M\"unster: Dynamics–Geometry–Structure and by
CRC 1442 Geometry: Deformations and Rigidity.
}
\date{\today}
\title{On the model theory of the Farey graph}
\begin{abstract}
 We axiomatize the theory of the Farey graph and prove that it is $\omega$-stable of Morley rank $\omega$.
\end{abstract}

\maketitle
\section{Introduction}

Our purpose in this article is to investigate the model-theoretic properties of the Farey graph, 
a combinatorial and geometric structure that arises in a number of mathematical fields. 
In geometric group theory, the Farey graph appears e.g. in the study of the free factor complex and the curve graph.

The Farey graph is isomorphic to the complex of conjugacy classes of free factors of rank 2, $OF_2$ \cite{Free_factor_complex_rigidity}, which appears 
in any higher rank complex of conjugacy classes of free factors. This makes studying the theory of the Farey graph essential 
for studying the free factor complex and the complex of conjugacy classes of free factors from the model-theoretic perspective.

The main purpose of this article is to prove the following:

\begin{theorem} \label{main_th}
    The theory of the Farey Graph is axiomatized by
\begin{enumerate}
\item every vertex is contained in an edge;
\item every edge is contained in two triangles; and

\item every finite subgraph of cardinality at least two has two removable vertices\footnote{In fact, it suffices to ask for one removable vertex.}
\end{enumerate}
and   is $\omega$-stable of Morley rank~$\omega$.
\end{theorem}

Note that $\omega$-stability was already  proved  in \cite{modeltheorycurvegraph} for a much wider class of graphs, but without computing the Morley rank.
Our proof proceeds along the lines of the coresponding results for the free pseudospace \cite{TENT_2014} and open generalized polygons \cite{AmmerTent2024}.

\section{The theory of the Farey graph}

The Farey graph is a planar graph that can be defined in a number of different ways. We will use the following
combinatorial definition of the Farey graph, see e.g. \cite{office_hours} and \cite{KURKOFKA2021223}.

\begin{definition}
    The Farey graph $F_1=F_1(e_0)$ of level $1$ consists of two triangles sharing the edge $e_0$.
    Inductively, the Farey graph $F_{n+1}(e_0)$ of level $n + 1$ is the graph obtained from $F_n$ by adding a new vertex $v_e$ for every boundary edge $e$ of $F_{n}$,
    with new edges joining $v_e$ precisely to the vertices of $e$.

    Finally, the \textbf{Farey graph} is the union of these  graphs \[G_F =G_F(e_0):= \bigcup_{n \in \mathbb{N}} F_n.\]

\end{definition}

By construction, every edge in $G_F$ is contained in exactly two triangles and the automorphism group of the Farey graph acts transitively on the set of ordered triangles.

Note that the symmetry of $F_1(e_0)$ which swaps the two triangles extends to  a symmetry of $F_n, n\in N$  swapping the two hemispheres of $F_n(e_0)$ and fixing $e_0$.

We  recall some graph-theoretic notions. 
For $a$ and $b$ in a connected graph $A$, the distance $d(a, b)$ between $a$ and $b$ is the smallest number 
$m$ for which there is a path $a = a_0, a_1, ..., a_m = b$ with $a_i$ in $A$, where 
$a_i$ and $a_{i+1}$ are adjacent for $0 \leq i < m$ and we call a path witnessing this distance a \emph{geodesic}. A vertex $x\in A$ is a cut-vertex in $A$ if $A\setminus\{x\}$ is not connected.
%We may write $d_A(a, b)$ to emphasize the dependence on the graph $A$.

For a vertex $a \in A$ the valency of $a$ in $A$,
 $\val(a)=val_A(a)$, is the number of neighbours of $a$ in $A$.

It is easy to see from the construction that any finite subgraph $A$ of $G_F$ with $|A|\geq 2$ has (at least) two \emph{removable vertices}
in the sense of the following definition:

\begin{definition}\label{def:removable}
\begin{enumerate}
    \item 
    Let $A$ be a graph, $x\in A$ a vertex. We call $x$ removable from $A$ if $\val(x)$ is at most $1$ or if $\val(x)=2$ and the two neighbours of $x$ form an edge.

    \item If $A$ is a subgraph of $B$, we say that $A$ is strong in $B$ ($A\leq B$) if any non-empty finite subgraph $B_0$ of $B$ not contained in $A$ has a removable vertex $x\in B_0\setminus A$. 
    \end{enumerate}
\end{definition}

\begin{definition}\label{def:K}
    
Let $(\K,\leq)$ be the class of finite graphs $A$ such that
\begin{enumerate}
    \item every edge is contained in at most two triangles, and
    \item every subgraph $A_0$ of $A$ with $|A_0|\geq 2$ has at least two removable vertices.
\end{enumerate}
\end{definition}

\begin{lemma}\label{lem:FG in K}
Every finite subgraph of the Farey graph belongs to $\K$.
\end{lemma}
\begin{proof}
Let $A$ be a finite subgraph of the Farey Graph. Since  condition (1) is clear, it suffices to show that if $|A|\geq 2$, then $A$ has two removable vertice. We may assume that some edge $e$ of $A$ is contained in two triangles (otherwise the claim holds trivially). Suppose $k$ is minimal with $A\subset F_k(e)$. Then in each of the two hemispheres of $F_k(e)$ the last vertex of $A$ added in the construction of $F_k(e)$ is removable.
\end{proof}

\begin{remark} Note that $\emptyset\leq A$ for all $A\in \K$
and that $\leq$ is transitive on $\K$, i.e. if $A\leq B$ and $B\leq C$, then $A\leq C$.
\end{remark}

The construction of the Farey graph and these definitions immediately imply:

\begin{lemma}\label{lem:strong after removing point}
\begin{enumerate}
\item   If $A$ is strong in $B$, $x,y\in A$, then $d_A(x,y)=d_B(x,y)$.
\item   If $A, B\in \K$, then $A$ is strong in $B$ if and only if we can obtain $B$ from $A$ by successively adjoining removable vertices.
\item  Suppose $A\leq B$ and $x\in B$ is removable over $A$. If one of the following holds
\begin{itemize}
\item $x$ has exactly two neighbours in $A$;
\item $x$ has exactly one neighbour in $A$ and if $val_B(x)\geq 2$, then $x$ is a cut-vertex in $B$; or
\item the connected component of $x$  in $B$ does not intersect $A$.
\end{itemize}
then $A\cup \{x \}\leq B$.
\end{enumerate}

\end{lemma}

\begin{definition}\label{def:free amalgam}
For graphs $A, B, C$ with $A\subset B, C$, we denote by
\[B\otimes_A C\]
the free amalgam of $B$ and $C$ over $A$, i.e. the graph with vertex set $A\cup (B\setminus A)\cup (C\setminus A)$ with the induced edges on $B$ and $C$, respectively, and no edges between vertices of $B\setminus A$ and $C\setminus A$.
\end{definition}

The following is now obvious:
\begin{proposition}\label{prop:amalgamtion}
    The class $(\K,\leq)$ is closed under (AP). More precisely, for  $A, B, C\in \K$ with $A\leq B, C$ there exists $D\in K$ such that $B,C\leq D$ and the embeddings commute.
    
    We let $FG$ denote the Hrushvovki limit of this class.
\end{proposition}
\begin{proof}
Suppose $A, B, C\in \K$ with $A\leq B, C$. We may assume inductively that $|B\setminus A|=1$. Then the free amalgam $D=B\otimes _A C$ belongs to $\K$ unless the vertex $b\in B\setminus A$ has valency $2$ and the edge between the neighbours of $b$ is already contained in two triangles in $C$. In this case, we put $D=C$. Clearly, in either case we have $B, C\leq D$. 
\end{proof}

\begin{remark}
$FG$ and $G_F$ have unbounded diameter.
\end{remark}

We now give axioms for the theory $T=Th(G_F)$ of the Farey graph $G_F$ in the language $L = \{ E \} $ of graphs.

\begin{definition}
    Let $T$ be the theory of graphs axiomatized by 
\begin{enumerate}
   \item every vertex is contained in an edge,
    \item any edge is contained in exactly two triangles; and
    \item every finite subgraph of cardinality at least two has two removable vertices.
\end{enumerate}
\end{definition}
We claim that $T$ is in fact a complete theory.
Clearly, the Farey graph $G_F$ and  the Hrushovski limit $FG$ of $(\K,\leq)$ are models of $T$. 

\begin{remark}\label{rem:triangulated}
Note that (3) implies that any simple cycle in any model of $T$ is triangulated as all vertices in such a cycle have valency $2$.
\end{remark}

We first note the following:

\begin{remark}\label{rem: tree of Fareys}
Let $\G$ be a cycle free graph. Replace every vertex $x\in\G$ by a copy $F_x$ of the Farey graph. For every edge $(x,y)\in \G$, pick vertices $v_x\in F_x, v_y\in F_y$ and identify them. The resulting graph is a model of $T$. 
\end{remark}

We first establish:
\begin{proposition}\label{prop:K-saturated}
    A model  $M$ of $T$ is $\K$-saturated if and only if $\omega$-saturated.
\end{proposition}

\begin{proof}
First suppose $M$ is an $\omega$-saturated model of $T$. We have to show that  $M$ is $\K$-saturated, i.e. we have to show that if $A$ is a finite subgraph of $M$, then $A\in \K$ (which is clear from the axioms of $T$) and that for every pair $A, B\in \K$ with $A\leq B$ and  every strong embedding of $A$ into $M$ we can find a copy of $B$ over $A$ in $M$. (Note that we identify $A$ with its image in $M$.)

Clearly, it suffices to prove this inductively in the case where $B=A\cup\{x\}$ where $x$ is a removable vertex in $B$. 

First consider the case $\val(x)=0$. Note that an $\omega$-saturated model has infinitely many connected components. Since $A$ is finite, we can find a vertex $b\in M$ whose connected component does not intersect $A$. By Lemma~\ref{lem:strong after removing point} we see that if $A\leq M$, then also $A\cup\{ b\}\leq M$.

If $\val(x)=1$, let $a\in A$ be the unique neighbour of $x$. Since every vertex in $M$ has infinite valency and $M$ is $\omega$-saturated, the examples in Remark~\ref{rem: tree of Fareys} show that  we can find a neighbour $b\in M$ of $a$ which is a cut-vertex in $M$.  By Lemma~\ref{lem:strong after removing point} we see that if $A\leq M$, then also $A\cup\{ b\}\leq M$.

    If $x$ has valency $2$ with neighbours $y_1,y_2$ forming an edge in $A$, then this edge is contained in at most $2$ triangles in $B$. Since every edge of $M$ is contained in exactly two triangles, we can find the corresponding triangle in $M$ and we see as before that this extension of $A$ is strong in $M$.

    For the other direction, suppose that $M$ is $\K$-saturated. Then clearly, $M$ is itself a  model of $T$. If $M'$ is an $\omega$-saturated structure elementarily equivalent to $M$, then by the first direction we see that $M'$ is $\K$-saturated. Now $M$ and $M'$ are also partially isomorphic and so $M$ is $\omega$-saturated.
\end{proof}

\begin{corollary}\label{cor:th(FG)}
   $T$ is complete and, hence, is the theory of the Farey Graph.  
\end{corollary}
\begin{proof}
 Clearly, the Farey graph is a model of $T$. To see that $T$ is complete it  is enough to show that  any model of $T$ is elementarily equivalent to $FG$. So let $M\models T$ and let $M'\equiv M$ be $\omega$-saturated. Then by  Proposition~\ref{prop:K-saturated}, $M'$ is $\K$-saturated and so $M'$ and $FG$ are partially isomorphic and hence elementarily equivalent. 
\end{proof}

We can now classify the models of $T$ and we will show that every model is of the form as described in Remark\ref{rem: tree of Fareys} we introduce the following equivalence relation: 

\begin{definition}
    Let $M$ be a  model of $T$.
   Two edges in $M$ are \textbf{equivalent} if they generate the same Farey graph as a subgraph of $M$. 
\end{definition}

\begin{remark}
    Removing a single vertex from $G_F$ leaves the remaining graph connected. Hence for every $x, y \in G_F$, there is a simple cycle in $G_F$ containing $x$ and $y$. Thus, we could equivalently  define the equivalence relation on the edges by saying that two edges are equivalent if and only if they are contained in a simple cycle.
\end{remark}

\begin{lemma}\label{cluster_lemma}
    Two equivalence classes intersect at most in a single vertex.
    
\end{lemma}

\begin{proof}
    If $x$ and $y$ are vertices belonging to equivalence classes $E_1$ and $E_2$, then they
    are connected by paths in $E_1$ and in $E_2$. Either the paths agree, and
    then $E_1$ and $E_2$ contain a common edge and hence $E_1 = E_2$, or we obtain a simple
    cycle. The cycle is triangulated by Remark~\ref{rem:triangulated} and hence again $E_1$ and $E_2$ contain a common edge.
\end{proof}
 
By abuse of terminology we say that a vertex belongs to an equivalence class of edges if it is contained in one of the edges of the equivalence class.

\begin{definition}\label{def:tree}
For any model $M$ of $T$ we define a bipartite graph $G_{tree}(M)$ on the set of vertices of $M$
and the set $\E$ of equivalence classes of $M$, where edges between a vertex $x\in M$ and an equivalence class $E\in\E$ are given by containment,
i.e. a vertex is connected to an equivalence class of edges if it is contained in the equivalence class.
\end{definition}

\begin{lemma}
    For every model $M$ of $T$, $G_{tree}(M)$ is a forest.
\end{lemma}

\begin{proof}
    If the graph $G_{tree}(M)$ contains a cycle, this cycle translates into a cycle in $M$ (by connecting the vertices  of
    the bipartite graph inside their equivalence classes). Any cycle in $M$ is
    triangulated by Remark~\ref{rem:triangulated}, hence the cycle is contained in a single equivalence class.    
\end{proof}

We therefore established:
\begin{theorem}\label{thm:classification}
Every model of $T$ arises from a cycle free graph as described in Remark~\ref{rem: tree of Fareys}.
\end{theorem}

\section{Quantifier elimination}\label{sec:QE}

We now extend the language of graphs to a language $L'$ to obtain quantifier elimination for $T$.
We start with an easy lemma: 
\begin{lemma}\label{lem:finitely_many_path}
Let $M$ be a model of $T$.
    For  any two vertices $x, y\in M$ 
  there are only finitely many geodesics from $x$ to $y$.
\end{lemma}

\begin{proof}
    We prove the claim by induction on $d(x,y)=k$. For $k=1$ the claim is clear. Now let $k>1$ and assume the claim is proved for all
$k'<k$.

  Let $p_0=(x_0=x,\ldots, x_k=y)$ be a path and let $(p_1=y_0=x,\ldots,
y_k=y)$ be a path with $x_i\neq y_i, i=1\ldots k-1$. If no such path
exists we are done by induction.  The concatenation of $p_0$ and $p_1$
is a simple cycle. This cycle has at least two removable points and since
all vertices have valency $2$ and $d(x,y)=k$, the removable points must be
$x$ and $y$. Thus, $(x_0,x_1,y_1)$ form a triangle.
If there is another path $p_2=(z_0=x,\ldots, z_k=y)$ with $z_1\neq
x_1,y_1$, then by the same argument $(x_0,x_1,z_1)$ form a triangle.
Since the edge $(x_0,x_1)$ is contained in exactly two triangles, this
shows that any path from $x$ to $y$ of length $\K$ passes through one of
$\{x_1, y_1,z_1\}$ and the claim follows from the induction assumption.
\end{proof}

\begin{corollary}
        For any two vertices $x,y$ of the Farey graph the algebraic
        closure  of $x,y$ is the whole graph.
\end{corollary}
 \begin{proof}
     This follows directly from Lemma~\ref{lem:finitely_many_path} and the fact that the Farey graph is algebraic in any edge of the graph.
 \end{proof}

\begin{definition}\label{def:convex}
Let $M$ be a model of $T$. We identify the vertices of $M$ with the vertices of $G_{tree}(M)$ corresponding to those of $M$ and for every subset $B$ of $M$ we let $\conv(B)$ denote the convex closure of $B$ in $G_{tree}(M)$.

Let $\conv_M(B)\subset M$ be the set of all vertices in $G_{tree}(M)$  that belong to a vertex of type $\E$ in $\conv(B)$.
\end{definition}

\begin{corollary}\label{cor:gate}
    Let $M$ be a model of $T$, $x \in M$ and $B\neq\emptyset$ be a subset of $M$. 
    If $x \notin \conv_M(B)$, then there exists $y \in \conv_M(B)$ such that
    every path from $x$ to  any vertex in $\conv_M(B)$ passes through $y$. We call $y$ the gate for $x$ to $\conv_M(B)$. 
\end{corollary}
\begin{proof}
If $|B|=1$, then $\conv_M(B)=B$ and there is nothing to prove. So assume $|B|\geq 2$ and consider the shortest path from $x$ to $\conv(B)$. Since $G_{tree}(M)$ is a tree, there is a unique vertex $y$ in $\conv(B)$ such that for every $b\in \conv(B)$ any path from $x$ to $b$ passes through $y$. Clearly we must have $y\in \conv_M(B)$ and it follows that every  path from $x$ to $b\in\conv_M(B)$ also passes through $y$.
\end{proof}

We can now describe the algebraic closure in $T$. Note that it suffices to describe $\acl(A)$ for finite subsets $A$ of models of $T$. 
\begin{proposition}\label{prop:acl}
Let $M$ be a model of $T$ and let $A\subseteq M$ be a finite subset. Then $\acl(A)=\conv_M(A)$.
\end{proposition}
\begin{proof}
We may assume that $M$ is $\omega$-saturated, so without loss of generality $M=FG$. Note that $FG$ is countable, has countably many connected components and every connected component of $G_{tree}(M)$ is a regular tree with infinite valencies.

First suppose that $A=\{a\}$ is a single point. Since $\Aut(FG)$ permutes the connected components, clearly $\acl(a)$ is contained in the connected component of $a$.
Suppose there is some $x\in\acl(a), x\neq a$. Then by Lemma~\ref{lem:finitely_many_path}  there is an edge $e$ containing $a$ in $\acl(a)$ and hence $G_F(e)\subset \acl(a)$ contradicting the fact that  $\Aut(G_F)$ acts transitively on the infinitely many edges containing $a$.

Now suppose $|A|>1$. By the previous argument we may assume that $A$ is contained in a single connected component of $FG$.  Let $\conv(A)$ be the convex hull of $A$ in $G_{tree}(M)$. Since between any two elements of $A$, the set of shortest paths between them is in the algebraic closure of $A$ and each equivalence class of edges is algebraic in any single edge of the class, we see that $\conv(A)\subseteq\acl(A)$.

Now we claim that $\acl(A)\subseteq\conv(A)$. Let $x\notin\conv(A)$. Then by Corollary~\ref{cor:gate} there is a unique vertex $a\in\conv_M(A)$ such that every path from $x$ to $A$ passes through $x$. If $x\in\acl(A)$, then again by Lemma~\ref{lem:finitely_many_path} there is an edge $E(a,y)$ in $\acl(A)$ containing $a$ and not contained in $\conv(A)$. Hence the equivalence class of $E(a,y)$ is algebraic in $A$. However,  the subgroup of $\Aut(M)$ fixing $A$ acts transitively on the infinitely many edges containing $a$ in the Farey graph corresponding to $E(a,y)$. This is a contradiction.
\end{proof}

It is natural to expand the language of graphs to a language containing
binary predicates $d_n(x,y)$ expressing that $d(x,y)=n$. However, these
predicates are not enough to obtain quantifier elimination: Suppose $x$
has distance $n$ to $y$ and $y'$ where $y$ is in the same equivalence
class and $y'$ is in a different equivalence class. Then there is a
simple cycle containing $x,y$, but no simple cycle containing $x,y'$.

Furthermore, not all cycles of a given length have isometric triangulations.

It turns out that these are (essentially) the only obstruction to quantifier elimination:

\begin{definition}\label{def:min cycle} We say that a simple cycle $C$ is minimal if it has exactly two removable vertices. If $x,y$ are the removable vertices in $C$, then we say that $C$ is minimal for $x,y$.
\end{definition}

For $n\geq 1$, let $\C_n$ be the set of all  isomorphism types of minimal triangulated cycles $C$ in the Farey
Graph in which the removable vertices have distance $n$ in $C$ and
put $\C=\bigcup\C_n$.

For $C\in\C$ we let $P_C$ be the binary predicate expressing that $P_C(x, y)$ holds
if there exists a  triangulated cycle $C'$  isomorphic to $C$ and minimal for $x,y$.

For any sequence $\delta = (C_1, \dots C_m), C_i\in\C_{k_i}, i\leq m$,
we expand the language by binary predicates $P_\delta$ expressing that $P_\delta(x,y)$ holds if and only
if the following conditions hold:

\begin{enumerate}
    \item There exist connecting points $z_0=x,z_2, \dots, z_m=y$  such that
     $ P_{C_i}(z_{i-1}, z_i)$ holds for all $i=1,\ldots m$, and
    \item $d(x,y)=\sum_{i=1}^m k_i$.
\end{enumerate}

For $\delta$ the empty sequence we put $P_\d(x,y)$ if and only if $x=y$.

Note that by (2) we have $d(z_i,z_{i-1})=k_i$  for all $i=1,\ldots m$.

In other words, if $P_C(x,y)$ holds for some $C\in\C_k$, then  then $x,y$ lie in the same equivalence class and have distance $k$. However, there may exist smaller cycles $C_1,C_2\in \C$ such that also $P_{C_1,C_2}(x,y)$  holds. 

\begin{remark}\label{rem:unique connecting points}
Let $M$ be a model of $T$ and suppose that $m$ is minimal such that for some $\delta=(C_1,\ldots, C_m)$ we have $P_\d(a,b)$. Then the connecting points $z_0=a,\ldots, z_m=b$ are uniquely determined and are the vertices of type $M$ in $G_{tree}(M)$ on the shortest path from $a$  to $b$, i.e. $z_0,\ldots, z_m\in\dcl(a,b)$.
If $m$ is not necessarily minimal, then by Lemma~\ref{lem:finitely_many_path} the connecting points are in $\acl(a,b)$
\end{remark}

We also define ternary relation symbols $D_{n,\sigma}$ for $n\geq 1,\ \sigma\in Sym(|F_n|)$, in the following way:

For every $n\geq 1$, we fix an enumeration $v_1,\ldots, v_{|F_n|}$ of the Farey graph of level~$n$. 
Then $D_{n,\sigma}(x,y,z)$ holds if and only if there are $x_4,\ldots x_{|F_n|}$ such that the map
\[x\mapsto v_1, y\mapsto v_2, z\mapsto v_3\mbox{\  and \ }x_i\mapsto v_{\sigma(i)},\quad i=4,\ldots |F_n|,\] is an isometry.\footnote{We allow equality between the variables $x, y, z$ in which case the identification has to be adjusted accordingly.}

Note that if $D_{n,\sigma}(a,b,x)$ holds, then $x\in\acl(a,b)$ in any model of $T$.

We now define ternary predicates $Y_{\d_1,\d_2,\d_3,n,\sigma}$
for finite (possibly empty) sequences $\d_1,\d_2,\d_3$ in $\C$
where $Y_{\d_1,\d_2,\d_3,n,\sigma}(x,y,z)$ holds
if and only if
\begin{enumerate}

\item there exist witnesses $x',y',z'$ such that $D_{n,\sigma}(x',y',z')$ holds, and
\item $P_{\d_1}(x,x'), P_{\d_2}(y,y')$ and $P_{\d_3}(z,z')$ hold.\begin{comment} and 
\item 
\begin{align*} d(x, y) &= d(x, x') + d(x', y') + d(y', y)\\
d(x, z) &= d(x, x') + d(x', z') + d(z', z) \mbox{ and }\\
d(y, z) &= d(y, y') + d(y', z') + d(z', z).
\end{align*}\end{comment}
\end{enumerate}

We call $|\d_1|+|\d_2|+|\d_3|$ the \emph{total length} of $\epsilon$. If the total length of $\epsilon$ is zero, then $Y_\epsilon=D_{n,\sigma}$.

\begin{remark}\label{rem:Y algebraic}
Note that if $\epsilon=(\d_1,\d_2,\d_3,n,\sigma)$ has minimal total length such that $Y_\epsilon(x,a,b)$ holds for some witnesses $x',a',b'$, then by the uniqueness from Remark~\ref{rem:unique connecting points} and the algebraicity of $D_{n,\sigma}(x',a',b')$  we see that $x'$ is the gate for $x$ to $\acl(a,b)$.
In particular, if $\d_1=\emptyset$, then $x\in\acl(a,b)$.
 
\end{remark}

We let $L'$ denote the language of graphs expanded by the binary predicates $P_\delta$ and the ternary predicates $Y_\epsilon$ for all  $\epsilon=(\d_1,\d_2,\d_3,n,\sigma)$ where $\d,\d_1,\d_2,\d_3$ are finite sequences in $\C$.

\begin{theorem}\label{thm:QE}
$T$ has quantifier elimination in the language $L'$.
\end{theorem}

We will show that every type is determined by its quantifier-free part.
We start with the following easy (and well-known) lemma:

\begin{lemma}\label{lem:stabilizer}
Let $X=\{x_1,x_2,x_3\}$ form a triangle in the Farey graph $G_F$. Then the pointwise stabilizer of $X$ in $Aut(G_F)$ is trivial and hence $Aut(G_F)$ acts regularly on the set of ordered triangles.
\end{lemma}
\begin{proof}
Since $Aut(G_F)$ is transitive on edges, the transitivity on triangles follows easily.
    Since every edge is contained in exactly two triangles, fixing one of them implies that the other one is fixed as well. Now the claim follows inductively.
\end{proof}

\begin{corollary}\label{cor:rigid}
    Let $S_1, S_2$ be  finite connected triangulated subgraphs of the  Farey Graph without cut-points and without vertices of valency at most $1$. If $S_1, S_2$ are isometric, there is a unique automorphism of the Farey Graph taking $S_1$ to $S_2$.
\end{corollary}
\begin{proof}
By Lemma~\ref{lem:stabilizer} there is a unique automorphism of the Farey Graph taking the triangle corresponding to a removable vertex in $S_1$ to that in $S_2$. Since every edge is contained in exactly two triangles, this automorphism takes $S_1$ to $S_2$.
\end{proof}
If $S_1, S_2$ have cut-points, the proof shows that there is at most one such automorphism taking $S_1$ to $S_2$.

We now prove Theorem~\ref{thm:QE}:
\begin{proof}
Let $M$ be an $\omega$-saturated model of $T$, $A\subset M$ and $b\in M$. It is enough to show that the quantifier-free type of $\tp_{qf}(b/A)$ in the language $L'$ implies $\tp(b/A)$.

For this we show the following: if  $\tp_{qf}(b/A)=\tp_{qf}(b'/A)$ for some $b'\in M$, then there is an automorphism of $M$ fixing $A$ pointwise and taking $b$ to $b'$.

If $\tp_{qf}(b/A)$ does not contain any formula of the form $P_\d(x,a)$, then no element of $A$ lies in the same connected component as $b$ and $b'$. By homogeneity of the Farey graph and the classification of models of $T$ there is an automorphism pointwise fixing all connected components of $M$ containing an element of $A$ and taking $b$ to~$b'$.

Now suppose that $b\in\acl(A)$. Since $\acl(A)=\conv_M(A)$,  we see that $\tp_{qf}(b/A)$ contains $Y_\epsilon(a,a',x)$ for some $a,a'\in A$ and $\epsilon=(\d_1,\d_2,\emptyset,n,\sigma)$ of minimal total length. Let  $m$ be minimal such that there is $\d$ of length $m$ with $P_\d(a,a')$. Then $\d=\d_1\circ(C)\circ\d_2$ for some $C\in\C$ where $\circ$ denotes concatenation of sequences. Let $z_0=a,\ldots, z_m=a'$ be the connecting points for $P_\d(a,a')$. These are unique
by Remark~\ref{rem:unique connecting points}. Then $b,b'$ lie in the same Farey graph vertex of $G_{tree}(M)$ 
as $z_{|\d_1|}, z_{|\d_1|+1}$. By Corollary~\ref{cor:rigid} there is an automorphism fixing $A$ pointwise and taking $b$ to~$b'$.

If $b\notin\acl(A)$ and  $a\in A$ is the gate for $b$ to $\acl(A)$, then $\tp_{qf}(b/A)$ contains a formula $P_\d(x,a)$ for some $\d=(C_1,\ldots, C_m)$ with $m$ minimal. By Remark~\ref{rem:Y algebraic} $\tp_{qf}(b/A)$ contains no formula  $Y_\epsilon(x,a,a')$ for $a,a'\in A$ where $\epsilon=(\d_1,\d_2,\emptyset,n,\sigma)$ for sequences $\d_1,\d_2$ in $\C$.

We now do induction on $m$. If $m=1$, then $P_C(x,a)$ holds for some $c\in\C$. Since  $\tp_{qf}(b/A)$ does not contain any formula of the form $Y_\epsilon(a,c,x)$ for $a,c\in A$, this implies that $b$ and $b'$ lie in an equivalence class at distance $1$ from $a$ in $G_{tree}(M)$. There is an automorphism fixing $A$ pointwise and taking the equivalance class containing $b$ to the one containing $b'$. Now the claim follows from Corollary~\ref{cor:rigid}.

Suppose the claim has been proved for all $m'<m$ and suppose $P_\d(b,a)$ has connecting points $a=z_0,\ldots, z_m=b$ and $P_\d(b',a)$ has connecting points $a=z'_0,\ldots, z'_m=b'$. Then by induction assumption there is an automorphism fixing $A$ pointwise and  taking  $z_0,\ldots, z_{m_1}$ for $P_\d(b,a)$ to  $z'_0,\ldots, z'_{m_1}$. Let $b''$ be the image of $b$. Again by the induction assumption for $m=1$ we can compose this with the automorphism fixing $M\cup\{z_0',\ldots, z_{m-1}'\}$ pointwise and taking $b''$ to $b'$.

Finally suppose that the gate for $b$ in $\acl(A)$ is in $\acl(A)\setminus A$. Then there is $\epsilon=(\d_1,\d_2,\d_3,n,\sigma)$ of minimal total length such that $Y_\epsilon(a,c,b)$ with $a,c\in A$ and witnesses  $a',c',\tilde{b}$. Then $\tilde{b}\in\acl(a,c)$ is the gate for $b$ to $\acl(A)$. Let $\tilde{b'}$ be such that $a',c',\tilde{b'}$ witness $Y_\epsilon(a,c,b')$. By the first part of the argument, there is an automorphism fixing $A$ pointwise and taking $\tilde{b}$ to $\tilde{b'}$ and $b$ to some $b''$.
Since $P_{\d_3}(b'',\tilde{b'})$, we can now apply the argument from the previous paragraph to find an automorphism fixing $A\cup\{\tilde{b'}\}$
and taking $b''$ to $b'$.

\end{proof}
   \section{$\omega$-Stability and Morley Rank}

As mentioned above,   $\omega$-stability was proved in a much more general setting in~\cite{modeltheorycurvegraph}. They also mention
    that the Morley rank of the Farey Graph is at least $\omega$ and the set of vertices at distance $n$ from a vertex $a$ has Morley rank at least $n$. Having quantifier elimination in  the languate $L'$ now allows us to establish:
    
 \begin{definition}
Let $\d,\d'$ be finite sequences in $\C$. We say that $\d'$  \emph{embeds } into $\d$ if whenever $P_\d(x,y)$ holds, then so does $P_{\d'}(x,y)$.
\end{definition}
Clearly, if $\d'$ embeds into $\d$, then $|\d|\leq|\d'|$.

 \begin{theorem}\label{thm:omega}
     The Morley rank of the theory of the Farey graph is $\omega$.
 \end{theorem}

We start with the following easy lemma:

\begin{lemma}\label{lem:finitely many cycles}
Let $M$ be a model of $T$ and let $(x,y)$ be an edge, $3\leq k<\omega$. There are finitely many simple $k$-cycles  containing $(x,y)$.
\end{lemma}
\begin{proof}
Clearly the claim holds for $k=3$.
Now suppose $k\geq 4$ is minimal such that there are infinitely many simple $k$-cycles  containing $(x,y)$. Since every cycle has at least two removable vertices and $x$ and $y$ cannot both be removable, in each cycle there is a removable vertex different from $x, y$. Removing it yields a simple cycle of length $k-1$ containing $(x,y)$. Thus there is a $k-1$-cycle containing $(x,y)$ which arises in this way from infinitely many $k$-cycles. Since every edge is contained in exactly two triangles, this is impossible.
\end{proof}

\begin{lemma}\label{lem:str min}
Let $a,b\in G_F, k=d(a,b),$ and $C\in \C_n, C'\in\C_m$. Then the formula $P_C(x,a)\wedge P_{C'}(x,b)$ is algebraic.
\end{lemma}
\begin{proof}
By Lemma~\ref{lem:finitely_many_path} there are only finitely many geodesics from $a$ to $b$. By Lemma~\ref{lem:finitely many cycles} there are only finitely many cycles of type $C$ and type $C'$ containing any edge of one of these paths from $a$ to $b$. Now consider intersecting cycles of type $C$ containing $a$ and of type $C'$ containing $b$ and not containing any edge from any of the paths from $a$ to $b$. Since they intersect at distance $\leq n$ from $a$ and distance $\leq m$ from $b$, these cycle segments together with one of the geodesics from $a$ to $b$ form a cycle of length at most $k+m+n$. Again there are at most finitely many altogether. Applying Lemma~\ref{lem:finitely many cycles} again shows that there are only finitely many cycles of type $C$ at $a$ and of type $C'$ at $b$ which intersect, and hence the formula is algebraic.

\end{proof}

 \begin{proof}[Proof of Theorem~\ref{thm:omega}]

By quantifier elimination in the language $L'$, the definition of the predicates $Y_{\delta_1, \delta_2, \delta_3, n, \sigma}$ as a conjunction of instances of formulas $P_\delta$ and Remark~\ref{rem:Y algebraic} it is enough to prove that for any $a\in M$ the Morley rank of     $P_\delta(x,a)$ is exactly $m=|\delta|$. 

    Let $M$ be an $\omega$-saturated model of $T$, so $M$ has infinitely many connected components and every vertex in $G_{tree}(M)$ has infinite valency.

The proof is by 
 induction on $m$.  

\bigskip

Let $m=1$ and  let $F_i, i<\omega$, be  neighbours of $a$ in $G_{tree}(M)$. Since all $F_i$ are isomorphic to the Farey graph, each contains a vertex $b_i$ such that $P_\d(a, b_i)$ holds, hence  $RM(P_\d(a, x)) ) \geq 1$ for all $a\in M$.

We claim that $P_\d(x,a)$ is strongly minimal. It suffices to prove that for all sequences $\d'$ and all $b\in M$ the formula $P_{\d'}(x,b)\wedge P_\d(x,a)$ is either algebraic or cofinite in $P_\d(x,a)$.

So let $|\d'|=m'$ be minimal such that $P_\d(a,x)\wedge P_{\d'}(x,b)$ holds for infinitely many $x$. Then $m'\geq 2$ by Lemma~\ref{lem:str min}.
If $a, b$ lie in a different equivalence classes, then by Remark~\ref{rem:unique connecting points} the set of solutions for $P_\d(a,x)\wedge P_{\d'}(x,b)$ is the same as the set of solutions $P_\d(a,x)\wedge P_{\d''}(x,c)$ where $c$ is the gate for $b$ to the equivalence class of $a$ and $\d''$ is a proper end segment of $\d'$. This contradicts the minimality of $m'$. So $a, b$ lie in the same Farey graph.

If $\d'=\d_1\circ\d_2$, $P_{\d_1}(a,b)$ and $\d_2$ embeds into $\d$, then clearly $P_\d(x,a)\wedge P_{\d'}(x,b)$ is confinite $P_\d(x,a)$. 

If $P_{\d_1}(a,b)$ and $\d_2$ does not embed into $\d$, then $P_{\d_2}(x,a)\wedge P_\d(x,a)$ is finite by minimality of $m'$. 

Hence we now assume we cannot write $\d'=\d_1\circ\d_2$ such that $P_{\d_1}(a,b)$ holds.

Consider the finitely many geodesics from $a$ to $b$ and write $\d'=(C)\circ\d''$. 
There are at most  finitely many copies of $\d$ and of $C$ intersecting one of these geodesics by Lemma~\ref{lem:finitely many cycles}. 
If all copies of $C$ intersect one of these geodesics, we could replace $P_{\d'}(x,b)$ by $P_{\d''}(x,c)$ where $c$ is the corresponding connecting point, contradicting minimality of $m'$.
Hence infinitely many copies of $C$ do not intersect any of these geodesics.
Any realization of $P_\d(a,x)\wedge P_{\d'}(x,b)$ arises from a unique copy of $\d$ and a unique copy of $\d'$. Since these copies form a simple cycle of bounded length containing a geodesic from $a$ to $b$, there are only finitely many copies of $C$ involved. Hence we can again replace $P_{\d'}(x,b)$ by $P_{\d''}(x,c)$ where $c$ is the corresponding connecting point, contradicting minimality of $m'$.

Now suppose the claim is proven for all $m'<m$ and suppose $\d=(C_1,C_2)\circ\d'$ where $|\d'|=m-2$. Let $b_i\in F_i$ be such that $P_C(a,b_i)$ holds. Then by induction assumption, $P_{C\circ \d'}(x,b_i)$ has Morley rank at least $m-1$. Since the $b_i$ lie in distinct copies of the Farey graph, and any path from $b_i$ to $b_j$ passes through $a$, the sets $P_{\d'}(x,b_i)$ and $P_{\d'}(x,b_j)$ intersect at most in the set $P_{\d'}(x,a)$. Again by induction assumption these sets have Morley rank $m-2$, showing that $P_\d(x,a)$ has Morley rank at least $m$.

We now show that for any $\C$-sequence $\d'$ and vertex $b$, the Morley rank of either $P_{\d'}(x,b)\wedge P_\d(x,a)$ or  $\neg P_{\d'}(x,b)\wedge P_\d(x,a)$ is strictly less than $m$.

Let $m'$ be minimal such that for some $\d'$ of length $m'$  this is not the case.

If $a,b$ lie in different equivalence classes, the result follows from the uniqueness in Remark~\ref{rem:unique connecting points} and the induction assumption.

If $\d'=\d_1\circ\d_2$, $P_{\d_1}(a,b)$ and $\d_2$ embeds into $\d$, then  $\neg P_{\d'}(x,b)\wedge P_\d(x,a)$ has Morley rank strictly less than $m$.

If $P_{\d_1}(a,b)$ and $\d_2$ does not embed into $\d$, then $P_{\d_2}(x,a)\wedge P_\d(x,a)$ has smaller Morley rank by minimality of $m'$. 

Hence we now assume we cannot write $\d'=\d_1\circ\d_2$ such that $P_{\d_1}(a,b)$ holds.

Consider the finitely many geodesics from $a$ to $b$ and write $\d'=(C)\circ\d''$. Now we argue exactly as in the case $m=1$. This concludes the proof.
\end{proof}

To finish our description of the model theory of the Farey graph we characterize forking independence in $T$ as in \cite{TZ} Theorem 8.5.10.

\begin{proposition} [cp. Theorem 2.31 \cite{AmmerTent2024} and Theorem 2.35 \cite{TENT_2014}]\label{prop:forking}
The following are equivalent:
    \begin{enumerate}
        \item $\indep{B}{A}{C}$.
        \item $\acl(ABC) \cong \acl(AB) \otimes_{\acl(A)} \acl(AC)$.
        \item Every path from $\acl(B)$ to $\acl(C)$ passes through $\acl(A)$. 
    \end{enumerate}

\end{proposition}

\begin{proof}
    $(2) \Rightarrow (3)$:
    If $acl(ABC) \cong acl(AB) \otimes_{acl(A)} acl(AC)$, then, by the definition of the free amalgam,
every path from $\acl(B)$ to $\acl(C)$ passes through \( acl(A) \). 
  $(3) \Rightarrow (2)$: Conversely, if every
 path from $\acl(B)$ to $\acl(C)$ passes through $\acl(A)$, then 
$\acl(AB) \otimes_{\acl(A)} \acl(AC)$ is convex and hence equal to $\conv(ABC)=\acl(ABC)$ by Proposition~\ref{prop:acl}.

$(1) \Leftrightarrow (2)$:
We prove that Condition (2) has the following properties:
    \begin{itemize}
        \item
        (Invariance) Clearly Condition (2) is invariant under $Aut(M)$.
        \item 
        (Local character) We need to show that for all $A \subseteq M$ finite and $C \subseteq M$ arbitrary, there is some finite set $C_0 \subseteq C$ such that every path from $A$ to $C$ passes through $C_0$. By Corollary~\ref{cor:gate} there is a finite set $C_0\subset C$ such that every path from a vertex in $\acl(A)$ to a vertex in $\acl(C)$ passes through $\acl(C_0)$.

        \item
        (Weak Boundedness)  
        We need to show that for every $b\in M$ finite and $A \subseteq M$ arbitrary, 
        there are at most countably  types $\tp(b'/C)$ for $b'\in M$ such that $\tp(b/A)=\tp(b'/A)$ and that every path from $b'$ to $\acl(C)$ passes through $\acl(A)$.
        For $b \in acl(A)$, the claim is obvious.
Otherwise there is a unique gate $a\in\acl(A)$ for $b$ to $\acl(A)$. If $a$ is the gate for $b'$ to $\acl(AC)$, then the type of $b'$ over $A$ implies the type of $b'$ over $C$, hence there is only one type.
        
        \item 
        (Existence) We need to show that for all $b\in M$  and $A \subseteq C \subseteq M$ arbitrary, there is some $b'$ such that $tp(b/A) = tp(b'/A)$ and $\indep{b'}{A}{C} $.
      If the connected component of $b$ does not intersect $A$, then $b$ is already as required. Otherwise we let $a\in \acl(A)$ be the gate for $b$ to $\acl(A)$. By saturation there exists a neighbour $F$ of $a$ in $G_{tree}(M)$ not contained in $\acl(C)$. The type of $b$ over $\acl(A)$ is isolated by a formula $P_\d(x,a)$ and hence this type is realized in $F$.
        \item 
        (Transitivity) We have to show that if every path from $\acl(A)$ to $\acl(C)$ passes through $\acl(B)$  and every path from $\acl(A)$ to $\acl(D)$ passes through $\acl(BC)$,  then also every path from $\acl(A)$ to $\acl(D)$ passes through $\acl(B)$ and this is obvious.

        \item 
        (Weak Monotonicity) We have to show that if every path from $\acl(A)$ to $\acl(CD)$ passes through $\acl(B)$, then every path from $\acl(A)$ to $\acl(C)$ passes through $\acl(B)$ and every path from $\acl(A)$ to $\acl(D)$ passes through $\acl(BC)$. This is again obvious.   \end{itemize}
\end{proof}

\section{Acknowledgement}
We thank Connor Lockhart for pointing out the omission of the first axiom and Matteo Bisi for noting an error in the first version of this note.

\end{document}